\documentclass[12pt]{article}
\usepackage{amsmath, amstext, amsgen, amsbsy, amsopn, amsfonts, amssymb, graphicx,psfrag}
\usepackage{pdfsync}
\usepackage{epstopdf}
\usepackage{color}
\usepackage{amsthm}

\newtheorem{theorem}{Theorem}
\newtheorem{proposition}[theorem]{Proposition}
\newtheorem{lemma}[theorem]{Lemma}
\newtheorem{corollary}[theorem]{Corollary}

\newtheorem{definition}{Definition}

\title
{Remarks on Suzuki's Knot Epimorphism Number}
\author{
Jim Hoste\\
Pitzer College\\
\\
Joshua Ocana Mercado\\
Loyola Marymount University\\
\\
Patrick D.~Shanahan\\
Loyola Marymount University
}

\begin{document}

\maketitle

\begin{abstract}A partial order on prime knots can be defined by declaring $J\ge K$ if there exists an epimorphism from the  knot group of  $J$ onto the knot group of $K$. Suppose that $J$ is a 2-bridge knot  that is strictly greater than  $m$ distinct, nontrivial knots. In this paper we determine a lower bound on the crossing number of $J$ in terms of $m$. Using this bound we answer a question of Suzuki regarding the 2-bridge epimorphism number $\mbox{EK}(n)$ which is the maximum number of nontrivial  knots  which are strictly smaller than some 2-bridge knot with crossing number $n$.   We establish our results using techniques associated to parsings of a continued fraction expansion of the defining fraction of a 2-bridge knot.
\end{abstract}
\section{Introduction}

Given two knots $J$ and $K$ in $S^3$, an interesting question in knot theory, and one which has received a great deal of attention, is whether there exists an epimorphism from the fundamental group of the complement of $J$ onto the fundamental group of the complement of $K$. The existence of such an epimorphism defines a partial order on the set of prime knots and we write $J\ge K$ if such an epimorphism exists. The relation is clearly reflexive and transitive.  Proving it is antisymmetric is nontrivial. Suppose that $\phi: \pi_1(S^3-J)\to  \pi_1(S^3-K)$ and $\rho: \pi_1(S^3-K)\to  \pi_1(S^3-J)$ are epimorphisms. Then, the composition $\rho\circ\phi$ is an isomorphism because knot groups are Hopfian (see \cite{K}, Lemma~14.2.5). Hence $\phi$ is an isomorphism and $J=K$ because prime knots are determined by their knot groups \cite{W}.

It is easy to obtain examples where $J\ge K$. For example, if $J$ is a periodic knot with quotient knot $K$, then the quotient map induces the desired epimorphism. Torus knots provide special cases of this. For example, the $(2,15)$-torus knot $T(2,15)$ has periods of both 3 and 5, with quotients  $T(2,5)$ and $T(2,3)$, respectively. Note that in these examples, the crossing number of  $T(2,15)$ is 15, which is three times as big as the crossing number of  $T(2,5)$. If it were always the case that the crossing number of $J$ is at least 3 times the crossing number of $K$ whenever $J> K$, then this would provide a proof of {\sl Simon's Conjecture}, that a knot group can only map onto finitely many other non-trivial knot groups. While Simon's Conjecture is known to be true \cite{AY}, it is not true that the bigger knot must always have 3 times as many crossings as the smaller knot, for Kitano and Suzuki have shown that the 8-crossings knots $8_5, 8_{10}, 8_{15},8_{18}, 8_{19}, 8_{20}$ and $8_{21}$ are all greater than or equal to the trefoil knot $3_1$ \cite{KS}. However, these 8-crossing knots are all 3-bridge knots, and in  \cite{S}, Suzuki shows that if one restricts to the class of 2-bridge knots then the (strictly) bigger knot does indeed always have 3 times as many crossings as the smaller knot.

Focussing on the class of 2-bridge knots, Suzuki defines the {\it 2-bridge epimorphism number} $\text{\rm EK}(n)$ to be the largest number of distinct nontrivial knots which are strictly less than some 2-bridge knot with  crossing number $n$. An important result is that if $J\ge K$ and $J$ is a 2-bridge knot, then $K$ must also be a 2-bridge knot \cite{BBRW}. Thus, to compute $\text{\rm EK}(n)$ we need only count how many 2-bridge knots are smaller than each 2-bridge knot with crossing number $n$.  Examining all 2-bridge knots to 30 crossings, Suzuki determined that
\begin{equation}\label{Suzuki data}\text{\rm EK}(n)=\left \{\begin{array}{ll}
0&n=3,4,5,6,7,8\\
1&n=9,10,11,12,13,14,18,19,20,24\\
2&n=15,16,17,21,22,23,25,26,27,28,29,30.
\end{array}
\right .
\end{equation}
Because  the torus knot $T(2,45)$ is strictly larger than  $T(2,3), T(2,5), T(2,9)$, and $T(2,15)$, we have $\text{\rm EK}(45)\ge 4$. Suzuki then asked what happens between 31 and 45 crossings? How many crossings  must a 2-bridge knot  have  in order to be strictly larger than 3 or more nontrivial knots? In this paper we answer this question by proving the following theorem.

\begin{theorem}\label{main theorem} Suppose  $J$ is a 2-bridge knot which is strictly greater than $m$ distinct nontrivial knots. Then $J$ has at least $c_m$ crossings where $c_m$ is the smallest, positive, odd integer with at least $m$ positive, nontrivial, proper divisors.
\end{theorem}

\begin{table}[h]
\begin{center}
\begin{tabular}{c|cccccccccccccc}
$m$ & 1 & 2 & 3 & 4 & 5 & 6 & 7 & 8 & 9 & 10 & 11 & 12 &13 & 14  \\ \hline
$c_m$ & 9 & 15 & 45 & 45 & 105 & 105 & 225 & 315 & 315 & 315 & 945 & 945 & 945 & 945
\end{tabular}
\end{center}
\caption{Values of $c_m$ for $1 \le m \le 14$. }
\label{c_m values}
\end{table}

 Values of $c_m$ for small values of $m$ are given in Table~\ref{c_m values}. Thus, we can answer one of Suzuki's questions (Problem 4.6 of \cite{S}):  A 2-bridge knot must have at least 45 crossings in order to be strictly greater than three nontrivial knots. Interestingly, the answer is also 45 crossings in order to be strictly greater than four nontrivial knots. However, the required number of crossings for a 2-bridge knot to be strictly greater than five distinct nontrivial knots jumps to 105. Thus $\text{\rm EK}(45)=4$. More generally, we have the following corollary to Theorem~\ref{main theorem}.

\begin{corollary}\label{EK inverse}  The epimorphism number $\mbox{\rm EK}(c_m)=m$ if and only if $c_{m+1} > c_m$.
\end{corollary}

\begin{proof} The torus knot   $T(2, c_m)$ has crossing number  $c_m$  and is clearly greater than or equal to  $T(2,d)$ if $d$ is a divisor of $c_m$. Since $c_m$ has at least $m$ distinct proper divisors, it follows that  $\text{\rm EK}(c_m)\ge m$. On the other hand, if $J$ is a 2-bridge knot that is strictly greater than $m+1$  non-trivial, 2-bridge knots, then by Theorem~\ref{main theorem}  we have $\mbox{cr}(J) \ge c_{m+1}$. 
If $c_{m+1}>c_m$, then $\mbox{EK}(c_m) < m+1$ and so $\mbox{EK}(c_m)=m$. To prove the converse, first note that for all $m$, we have $c_{m+1}\ge c_m$, by the definition of $c_m$. Arguing by contradition, if $\mbox{EK}(c_m)=m$ and $c_m=c_{m+1}$, then $T(2, c_m)=T(2,c_{m+1})$ implies that $\mbox{EK}(c_m)\ge m+1$, a contradiction. 
\end{proof}

Theorem~\ref{main theorem}, its Corollary, and examples derived by  a construction explained in Section~\ref{seams} allow us to extend Suzuki's table of values of $\text{\rm EK}(n)$ for $n\le 45$. We postpone this discussion until Section~\ref{seams}.
Interestingly, $\text{\rm EK}$ is not an increasing, or even nondecreasing, function as the values given in (\ref{Suzuki data}) show.
However, we will prove the following theorem in Section~\ref{seams}. 
\begin{theorem}\label{lower bound theorem} For all $N\ge 3n\ge 9$, we have $\text{\rm EK}(N)\ge \text{\rm EK}(n)$. 
\end{theorem}
From this we obtain the following corollaries. In the first, the upper bound was previously shown in \cite{S}.
\begin{corollary}\label{almost increasing theorem} 
For all $n\ge 3$, we have $\text{\rm EK}(\lfloor \frac{n}{3} \rfloor)\le \text{\rm EK}(n) \le \lfloor \frac{n-3}{6}\rfloor$, where $\lfloor x\rfloor$ denotes the largest integer less than or equal to $x$.
\end{corollary}

\begin{corollary}
The function $\text{\rm EK}$ can take on any given value at most finitely many times.
\end{corollary}\label{EK cannot repeat infinitely}
\begin{proof} Let $k$ be any nonnegative integer. The torus knot $T(2, c_{k+1})$ is strictly greater than at least $k+1$ nontrivial knots and hence $\text{\rm EK}(c_{k+1})\ge k+1$. Now  $\text{\rm EK}(m)\ge k+1$ for all $m\ge 3 c_{k+1}$. Hence, the value of $k$ can only be taken on at most finitely many times.
\end{proof}   

\noindent Notice that Corollary~\ref{EK cannot repeat infinitely} implies $\displaystyle \lim_{n \to \infty}\text{\rm EK}(n)=\infty$. 

If  $J\ge K$  and $J$ is a 2-bridge knot then, as has already been  mentioned, $K$ must also be a 2-bridge knot \cite{BBRW}. Moreover, it is shown in this case (see \cite{A} and \cite{ALS})  that the epimorphism of fundamental groups is actually induced by a {\sl branched fold  map} on the complements of the knots as described by Ohtsuki, Riley, and Sakuma in \cite{ORS}. It is not necessary in this paper to describe their construction. Instead, we rely entirely on the results in ~\cite{GHS}, where a branched fold map between two 2-bridge knot complements is described entirely in terms of the continued fraction expansions associated to the two knots. This interpretation allows one to easily determine all 2-bridge knots that are smaller than a given 2-bridge knot. In the next section we review and build on the notation and main results of ~\cite{GHS}. In Section~\ref{main theorem section}, we  prove a few necessary facts about the function $c_m$ and then prove Theorem~\ref{main theorem}. In Section~\ref{seams} we prove Theorem~\ref{lower bound theorem} and determine $\text{\rm EK}(n)$ for $31\le n\le 45$.

This paper grew out of an undergraduate research project completed by Joshua Ocana Mercado that was directed by the third author and supported by the McNair Scholars Program \cite{OM}.

\section{Two-bridge Knots and Continued Fractions}

Recall that a 2-bridge knot is one having a 4-plat diagram as shown in Figure~\ref{4 plat}. Here a box labeled $a_i$  denotes $a_i$ right-handed half-twists if $a_i>0$, and $-a_i$ left-handed half-twists otherwise.  Note that by using $-a_i$ half-twists when $i$ is even produces an alternating diagram when all the $a_i$'s have the same sign.  Such a diagram is completely determined by the sequence $a_1, a_2, \dots, a_k$. 
\begin{figure}[htbp]
    \begin{center}
    \leavevmode
    \scalebox{.50}{\includegraphics{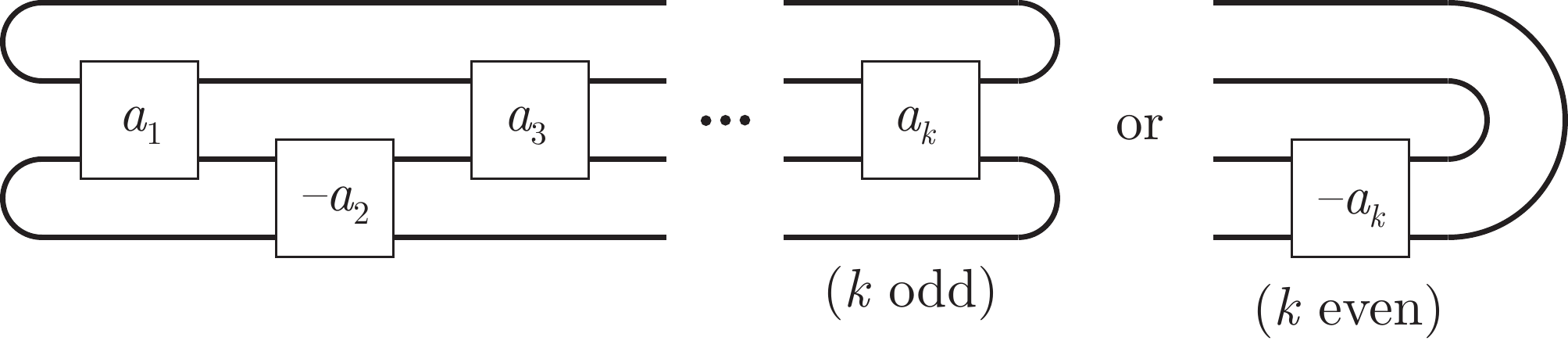}}
    \end{center}
\caption{The 2-bridge knot defined by the sequence $a_1, a_2, \dots, a_k$.}
\label{4 plat}
\end{figure}

If we form the continued fraction
$$p/q=[a_1, a_2, \dots, a_k]=\frac{1}{a_1+\displaystyle \frac{1}
{
\begin{array}{ccc}
a_2+&&\\
&\ddots&\\
&&+\displaystyle \frac{1}{a_k}\\
 \end{array}
 }
 }$$
then we may denote the knot as $K_{p/q}$. It is well known that  $K_{p/q}$ and $K_{p'/q'}$ are ambient isotopic as unoriented knots  if and only if $q'=q$ and $p' \equiv p^{\pm 1}\  (\mbox{mod } q)$ (see \cite{BZ:2003} for details).  In this paper, we will not distinguish between a knot $K_{p/q}$ and its mirror image $K_{-p/q}$.  Therefore, two 2-bridge knots  $K_{p/q}$ and $K_{p'/q'}$ are {\it equivalent} if and only if $q'=q$ and either $p' \equiv p^{\pm 1}\  (\mbox{mod } q)$ or $p' \equiv -p^{\pm 1}\  (\mbox{mod } q)$. It turns out that because the 4-plat diagram is of a knot, rather than a link, we must have $q$ odd. Furthermore, given any relatively prime pair of integers $p$ and $q$, with $q$ odd, and $-q<p<q$,  there is a 2-bridge knot with associated fraction $p/q$.

Any reduced fraction $p/q$ can be expressed as a continued fraction $r+[a_1, a_2, \dots, a_k]$ in infinitely many ways. However there are various schemes for producing a canonical expansion. The following Lemma is proven in \cite{GHS}.

\begin{lemma}
\label{special fraction} Let $\frac{p}{q}$ be a reduced fraction with $q$ odd. Then we may express $p/q$ uniquely as 
$$\frac{p}{q}=r+[a_1, a_2, \dots, a_k],$$
where each $a_i$ is a nonzero, even integer. 
 Moreover, $k$ must be even and $p$ and $r$ have the same parity. 
\end{lemma}

It is common to assume that  each {\it partial quotient} $a_i$ is not zero, however, we can easily make sense of continued fractions that use zeroes. A zero can be introduced or deleted from a continued fraction as follows:
$$[\dots, a_{k-2}, a_{k-1}, 0, a_{k+1}, a_{k+2}, \dots]=[\dots, a_{k-2}, a_{k-1}+a_{k+1}, a_{k+2},\dots].$$
Using this property, every continued fraction with all even partial quotients can be expanded so that each partial quotient is either $-2, 0$, or $2$. For example, a partial quotient of $6$ would be expanded to $2,0,2,0,2$ and $-4$ to $-2,0,-2$. This leads us to the following definition.

\begin{definition} {\rm Let ${\cal S}_\text{\rm even}$ be the set of all integer vectors $(a_1, a_2, \dots, a_k)$ such that
\begin{enumerate}
\item $k$ is even, 
\item each $a_i\in \{-2, 0, 2\}$, 
\item $a_1\ne0$ and $a_k\ne 0$,
\item if $a_i=0$ then $a_{i-1}=a_{i+1}\ne 0$.
\end{enumerate}
We call ${\cal S}_\text{\rm even}$ the set of {\it expanded even vectors of even length}.}
\end{definition}

We may  define an equivalence relation on ${\cal S}_\text{\rm even}$ by declaring that ${\bf a}, {\bf b} \in {\cal S}_\text{\rm even}$ are equivalent if ${\bf a}=\pm {\bf b}$ or ${\bf a}=\pm {\bf b}^{-1}$ where $-{\bf b}$ is obtained by negating every entry in $\bf b$, and ${\bf b}^{-1}$ is ${\bf b}$ read backwards.  We denote the equivalence class of ${\bf a}$ as $\hat{\bf a}$ and the set of all equivalence classes as $\hat{ \cal S}_\text{\rm even}$. 
The following proposition appears in \cite{GHS}.

\begin{proposition}
\label{2-bridge classification}
If $\Phi(\hat{\bf a})$ is defined to be the knot $K_{p/q}$ where $p/q=[{\bf a}]$, then 
$\Phi$   is a bijection between $\hat{\cal S}_\text{\rm even}$ and the set of equivalence classes of 2-bridge knots.
\end{proposition}

We will make use of the following two results from \cite{S}. 
If ${\bf a}\in {\cal S}_\text{\rm even}$, let  $\ell({\bf a})$ denote the {\it length} of $\bf a$ and  $\text{\rm cr}({\bf a})$ the  crossing number of $\Phi(\hat{\bf a})$. 
\begin{theorem}[Suzuki]\label{suzuki theorem} Suppose $\bf a \in {\cal S}_\text{\rm even}$. Then
\begin{enumerate}
\item the crossing number of $\Phi(\hat{\bf a})$ is equal to the sum of the absolute values of the components of $\bf a$ minus the number of sign changes in $\bf a$, and
\item $\ell({\bf a})+1\le \text{\rm cr}({\bf a}) \le 2 \ell({\bf a})$.
\end{enumerate}
\end{theorem}
\noindent Note that the second part of Theorem~\ref{suzuki theorem} follows immediately from the first part. 

The partial order on 2-bridge knots can be described entirely in terms of vectors in ${\cal S}_\text{\rm even}$. To do so, we  introduce some notation. First, if $\bf g$ and $\bf h$ are vectors, we denote their concatenation by $({\bf g}, {\bf h})$.  Next, if $c$ is an even integer, we define the vector $\bf c$ to be $(0)$ if $c=0$ and otherwise as $\pm (2,0,2,0,\dots, 2)$ where the sum of all the entries is $c$.  Ohtsuki, Riley and Sakuma show that  $J\ge K$,   if and only if there exist vectors ${\bf a}$  and $\bf b$, representing the knots $J$ and $K$, respectively, such that of $\bf a$ can be {\it parsed with respect to} $\bf b$, which means that $\bf a$ can be written as 
\begin{equation}\label{parsing}{\bf a}=({\bf b}, {\bf  c}_1, \epsilon_2 {\bf b}^{-1}, {\bf   c}_2, \epsilon_3 {\bf b}, {\bf   c}_3, \dots, \epsilon_n  {\bf b}),\end{equation}
where each $\epsilon_i$ is $\pm 1$ and each $c_j$ is an even integer. Moreover, we require that if $c_i=0$, then $\epsilon_i=\epsilon_{i+1}$. This statement does not require that $\bf a$ and $\bf b$ are in  ${\cal S}_\text{\rm even}$. The advantage of passing to expanded even vectors of even length is that parsings cannot be hidden by using the wrong vector. For example the knot $K_{38/85}$ is represented by the vector ${\bf a}=(2,4,4,2)$ which does not parse with respect to any vector. But, using ${\bf a}'=(2,2,0,2,2,0,2,2)\in {\cal S}_\text{\rm even}$ instead, reveals that  $K_{38/85}\ge K_{2/5}=\Phi((2,2))$.

In (\ref{parsing}), the vectors ${\bf   c}_i$ are called {\bf b}-{\it connectors} and  separate the {\it {\bf b}-tiles} $\epsilon_k {\bf b}^{(-1)^{k+1}}$. Note that $n$ must be odd and  we say that the parsing is an $n$-fold parsing.  
(See \cite{GHS} and \cite{ORS} for more details.)

In this paper, we  will be particularly interested in vectors of the form
\begin{equation}\label{two-connector alternating form}{\bf v}=({\bf a}, {\bf m}, {\bf a}^{-1}, {\bf n},{\bf a}, {\bf m}, {\bf a}^{-1}, {\bf n}, \dots , {\bf a}),
\end{equation} where ${\bf a} \in {\cal S}_\text{\rm even}$ and $m$ and $n$ are even integers. We call such a vector {\it two-connector alternating} and will denote it as ${\bf a}^{2p+1}_{m,n}$, where $\bf a$ appears $2p+1$ times.  If $\bf a$ is empty, then we prefer to write ${\bf a}^{2p+1}_{m,n}$ as $({\bf m},{\bf n})^p$ instead. Notice that when ${\bf a}$ is nonempty, $\bf v$ parses with respect to $\bf a$ in a special way---the only connectors are $\bf m$ and $\bf n$ which alternate in the parsing, and the $\bf a$-tiles are never negated. If $\bf v$ is a two-connector alternating vector, it may be possible to write $\bf v$ in the form given in (\ref{two-connector alternating form}) in more than one way. For example, if ${\bf a}=(2, 2)$,
${\bf b}=(2,2,0,2,2,4,2,2)$ and ${\bf c}=(2,2,0,2,2,4,2,2,0,2,2,4,2,2)$, then
$${\bf a}^{15}_{0,4}= {\bf b}^{5}_{0,4}={\bf c}^{3}_{0,4}.$$
Notice that ${\bf b}={\bf a}^3_{0,4}$ and that ${\bf c}={\bf a}^5_{0,4}$. Moreover, it is easy to see that
$$({\bf u}^{2p+1}_{m,n})^{2q+1}_{m,n}={\bf u}^{(2p+1)(2q+1)}_{m,n},$$ for all vectors $\bf u$ and even integers $m$ and $n$.
The following result is proven in \cite{GHS}.
\begin{theorem}[\cite{GHS}]\label{two-connector alternating} If $\bf v$ can be written in the form ${\bf v}={\bf a}^{2q+1}_{m,n}$, then $m$ and $n$ are unique. Moreover, there is a unique shortest vector $\bf g$ for which ${\bf v}={\bf g}^{2P+1}_{m,n}$  and  ${\bf a}={\bf g}^{2p+1}_{m,n}$ where $2P+1=(2p+1)(2q+1)$.
\end{theorem}

When a two-connector alternating vector is expressed as ${\bf g}^{2P+1}_{m,n}$, where $\bf g$ is of minimal length, we say that the expression ${\bf g}^{2P+1}_{m,n}$ is {\it generated} by $\bf g$.

The main result of \cite{GHS} is the following.

\begin{theorem}[\cite{GHS}]\label{GHS main theorem} 
 If $\bf c$ parses with respect to ${\bf a}_i$ for  all $1\le i\le m$, and   ${\bf a}_i$  does not parse with respect to ${\bf a}_j$ if  $i\ne j$  (in other words, the knots $\Phi({\bf a}_i)$ are pairwise incomparable), then there exists ${\bf g}\in {\cal S}_\text{\rm even}$, possibly empty, even integers $r$ and $s$, and  integers $p_i$ such that ${\bf a}_i={\bf g}^{2p_i+1}_{r,s}$ for each $i$. Moreover, if $2P+1$ is the least common multiple of the set $\{2p_i+1\}_{i=1}^m$, then ${\bf c'}={\bf g}^{2P+1}_{r,s}$ parses with respect to each ${\bf a}_i$ and no vector that parses with respect to each ${\bf a}_i$ is shorter than $\bf c'$.
\end{theorem}

Note that because of Theorem~\ref{two-connector alternating}, we may assume in Theorem~\ref{GHS main theorem}, that $\bf g$ generates each of the expressions ${\bf g}^{2P+1}_{r,s}$ and  ${\bf g}^{2p_i+1}_{r,s}$ for $1\le i\le m$.

\pagebreak\begin{lemma}\label{facts}\text{}
\begin{enumerate}
\item\label{g nonempty} If $P\in \mathbb N$, $m$ and $n$ are even integers, $\bf g$ is non empty, and   ${\bf g}^{2P+1}_{m,n}$ is generated by $\bf g$, then   ${\bf g}^{2P+1}_{m,n}$ parses with respect to $\bf b$ if and only if either ${\bf b}={\bf g}^{2q+1}_{m,n}$ and $2q+1$ divides $2P+1$, or $\bf g$ parses with respect to $\bf b$.
\item\label{g empty} If $m,n \in 2\mathbb Z-\{0\}$, $p\in \mathbb N$, and ${\bf b} \in {\cal S}_\text{\rm even}$, then $({\bf m},{\bf n})^p$ $d$-fold parses with respect to ${\bf b}$ if and only if ${\bf b}=({\bf m},{\bf n})^q$ and $2p+1=d(2q+1)$.
\end{enumerate}
\end{lemma}
\begin{proof} To prove item~\ref{g nonempty}, 
 suppose that $\bf g$ and $\bf b$ are incomparable, that is, neither parses with respect to the other. By \cite{GHS}, it follows that ${\bf g}={\bf f}^{2p+1}_{j,k}$ and ${\bf b}={\bf f}^{2q+1}_{j,k}$ for some vector ${\bf f}\in {\cal S}_\text{\rm even}$ and even integers $j$ and $k$. Because ${\bf g}^{2P+1}_{m,n}$ parses with respect to $\bf b$, and yet $\bf g$ and $\bf b$ are incomparable, we have that $\ell({\bf g})\ne \ell({\bf b})$. Assume that $\ell({\bf g})< \ell({\bf b})$. Comparing the beginning and end of the vector ${\bf g}^{2P+1}_{m,n}$ to the first and last ${\bf b}$-tile in its parsing with respect to $\bf b$ gives that $j=m$ and $k=n$. But now $\bf g$ is not a generator for the expression  ${\bf g}^{2P+1}_{m,n}$. If instead, $\ell({\bf g})> \ell({\bf b})$, then again we obtain $j=m$ and $k=n$ and again reach a contradiction. Thus $\bf g$ and $\bf b$ must be comparable.

 If   $\bf b$ parses with respect to $\bf g$, then because  ${\bf g}^{2P+1}_{m,n}$ parses with respect to $\bf b$, it follows that ${\bf b}={\bf g}_{m,n}^{2q+1}$ and $2q+1$ divides $2P+1$. If not, then $\bf g$ parses with respect to $\bf b$, as desired.

Item ~\ref{g empty} is simply the restatement of item~\ref{g nonempty} in the case where $\bf g$ is empty.
\end{proof}
 
\section{Proof of the Main Result}\label{main theorem section}
In this section we begin with a few  results regarding the  length of a vector and the function $c_m$ before proving Theorem~\ref{main theorem}.
If $\bf a$ admits a $d$-fold parsing with respect to $\bf b$, then it is a simple matter to compare their lengths and obtain the following result.

\begin{lemma}
Suppose that $\bf a, \bf b \in {\cal S}_\text{\rm even}$ and that $\bf a$ admits a $d$-fold parsing with respect to $\bf b$.(Note that this implies $d$ is odd.). Then $\ell({\bf a})\ge d\, \ell({\bf b})+d-1$.
\end{lemma}
\begin{proof} Suppose that 
$${\bf a}=({\bf b}, {\bf m}_1,\epsilon_2 {\bf b}^{-1}, {\bf m}_2, \dots, {\bf m}_{d-1},\epsilon_d {\bf b})$$
where each $m_i$ is even. Since each connector ${\bf m}_i$ has length at least 1, the result follows easily.
\end{proof}

\begin{definition}
For each natural number $m$, define $c_m$ to be the smallest, positive, odd integer having at least $m$ positive, nontrivial, proper divisors. If $m=0$, we define $c_0=3$ for convenience.
\end{definition}

We will need the following observations about $c_m$.
\begin{lemma}\label{facts about c_m}\hspace{1 in}
\begin{enumerate}
\item $c_m\le c_{m+1}$ for all $m\ge 0$.
\item $c_m\le 3 c_{m-1}$ for all $m\ge 1$.
\item For all natural numbers $r$ and $s$, $c_rc_s\ge c_{r+s+1}$.
\end{enumerate}
\end{lemma}
\begin{proof} If a positive odd integer has at least $m+1$ proper divisors, then clearly it has at least $m$ such. Hence, $c_m\le c_{m+1}$ for all $m>0$.  It is easy to see that $c_1=9$, so the result is also true when $m=0$.

Note that defining $c_0=3$ makes the second assertion a special case of the third, which we will now prove. 
If the prime factorization of $n$ is $n=p_1^{k_1}p_2^{k_2}\dots p_j^{k_j}$ then the total number of divisors of $n$ is $\displaystyle \prod_{i=1}^j (k_i+1)$. Because this depends only on the exponents $k_1, k_2, \dots, k_j$, and because $c_m$ is the smallest possible, positive, odd integer with at least $m$ positive, nontrivial, proper divisors, we see that the prime factorization of any $c_m$ must employ consecutive odd primes starting at 3.


Let $r$ and $s$ be any nonnegative integers and suppose the prime factorizations of $c_r$ and $c_s$ are
$$c_r=3^{a_1}5^{a_2}\dots p_j^{a_j} \text{ and } c_s=3^{b_1}5^{b_2}\dots p_k^{b_k}.$$ Without loss of generality, we may assume that $k\ge j$. Now the total number of divisors of $c_r\,c_s$ is 
$$\displaystyle \prod_{i=1}^j (a_i+b_i+1)\displaystyle \prod_{i=j+1}^k (b_i+1),$$
where the product from $j+1$ to $k$ is replaced with 1 if $j=k$. When $\displaystyle \prod_{i=1}^j (a_i+b_i+1)$ is multiplied out, there will be $3^j$ terms corresponding to the different ways in which one may choose one of the three summands from each factor. The terms can be placed in three sets, $R, S$, and $T$ as follows. The set $R$ consists of those terms where either $a_i$ or $1$ is chosen from each factor, the set $S$ consists of those terms where either $b_i$ or $1$ is chosen from each factor, and the set $T$ are all the remaining terms. The sets $R$ and $S$ have one term in common, namely $1=1\cdot 1\cdot \dots \cdot 1$. Let $\bar R, \bar S$ and $\bar T$ be the sums of all the terms in each of the sets $R, S$, and $T$ respectively. Thus

$$\displaystyle \prod_{i=1}^j (a_i+b_i+1)=\bar R+\bar S-1+\bar T.$$
But $\bar R=\displaystyle \prod_{i=1}^j (a_i+1)$ and $\bar S=\displaystyle \prod_{i=1}^j (b_i+1)$. Thus  the number of divisors of $c_r \, c_s$ is at least 
$$(r+2+s+2-1+\bar T)\displaystyle \prod_{i=j+1}^k (b_i+1)\ge r+s+3.$$ Hence $c_r \, c_s\ge c_{r+s+1}$. \end{proof}

We are now ready to prove our main result
\setcounter{theorem}{0}
\begin{theorem}\label{main theorem} Suppose  $J$ is a 2-bridge knot which is strictly greater than $m$ distinct nontrivial knots. Then $J$ has at least $c_m$ crossings where $c_m$ is the smallest, positive, odd integer with at least $m$ positive, nontrivial, proper divisors.
\end{theorem}
\setcounter{theorem}{13}

\begin{proof} Suppose $J=\Phi({\hat{\bf a}})$ is strictly greater than $m$ distinct nontrivial knots $K_1, K_2, \dots, K_m$. Because each $K_i$ must be 2-bridge, there exists vectors ${\bf b}_i\in {\cal S}_\text{\rm even}$ with $K_i=\Phi(\hat{\bf b}_i)$ for $1\le i\le m$. We will prove that $\ell({\bf a})\ge c_m-1$ which, when combined with Theorem~\ref{suzuki theorem}, will give the desired result.

 We proceed by induction on $m$. If $m=1$ and ${\bf a}$ admits a $d$-fold parsing with respect to  ${\bf b}_1$, then $d$ is at least $3$ and we have $\ell({\bf a})\ge 3\ell({\bf b}_1)+2\ge 3\cdot2+2\ge c_1-1.$

Assuming the result is true in the case of fewer than $m$ knots,  suppose now that $J$ is greater than $m$ distinct nontrivial knots $\{K_1, K_2, \dots, K_m\}$. 
 Let $A$ be the set of all $K_i$   such that  there does not exist $K_j$ with $J>K_j>K_i$.
 
\noindent{\bf Case I:} Suppose $A$ contains only one knot, say $K_1$.   By our inductive hypothesis, $\ell({\bf b}_1)\ge c_{m-1}-1$ and now $\ell({\bf a})\ge 3\ell({\bf b}_1)+2\ge 3 (c_{m-1}-1)+2\ge 3c_{m-1}-1\ge c_m-1$, using Lemma~\ref{facts about c_m}.

\noindent{\bf Case II:} Suppose $A$ contains two or more knots, say $K_1, K_2, \dots, K_n$ with $n>1$. It must be the case that $K_1, K_2, \dots, K_n$ are pairwise incomparable.  It now follows from Theorem~\ref{GHS main theorem} that there exists ${\bf g}\in {\cal S}_\text{\rm even}$, possibly empty, such that $K_i=\Phi({\bf g}^{2p_i+1}_{r,s})$ for some even integers $r$ and $s$ and nonnegative integers $p_i$ for $1\le i\le n$. Because these knots are incomparable, it follows that $2p_i+1\,|\,2p_j+1$ if and only if $i=j$. Let $2P+1=\text{lcm}(2p_1+1, 2p_2+1, \dots, 2p_n+1)$ and ${\bf a}'={\bf g}^{2P+1}_{r,s}$. If $\bf g$ does not  generate the expression ${\bf g}^{2P+1}_{r,s}$, then we may pass to the unique shortest vector that does. Hence, we may assume that $\bf g$ generates each of the expressions under consideration. It also follows from Theorem~\ref{GHS main theorem} that every vector in ${\cal S}_\text{\rm even}$ that parses with respect to ${\bf g}^{2p_i+1}_{r,s}$ for $1\le i\le n$ is at least as long as ${\bf a}'$. 


We now consider two cases: $\bf g$ is empty or not. Suppose first that $\bf g$ is empty. Rewriting the vectors under consideration, we have $K_i=\Phi(({\bf r},{\bf s})^{p_i})$ for $1\le i\le n$ and we let ${\bf a}'=({\bf r},{\bf s})^P$. Furthermore ${\bf a}'$ also parses with respect to every ${\bf b}_i$ for $n<i\le m$. By Lemma~\ref{facts}, we conclude that ${\bf b}_i=({\bf r},{\bf s})^{p_i}$ for $n<i\le m$ and that $2p_i+1\, | \,2P+1$. The integer $2P+1$ now has $m$ proper factors, $2p_1+1, 2p_2+1, \dots, 2p_m+1$, and hence $2P+1\ge c_m$. 
Thus 
$$
 \ell({\bf a})\ge \ell({\bf a}')\ge \ell(({\bf r},{\bf s})^P)
 \ge P(\ell({\bf r})+\ell({\bf s}))
 \ge 2P
 \ge c_m-1.
$$

Alternatively, suppose that $\bf g$ is nonempty. As before, ${\bf g}^{2P+1}_{r,s}$ parses with respect to each $\bf b_i$ for $n<i\le m$.  By Lemma~\ref{facts}, we conclude that for each $i>n$, either ${\bf b}_i={\bf g}^{2p_i+1}_{r,s}$ or that $\bf g$ parses with respect to ${\bf b}_i$. Assume that the former is true for $K_1, K_2,\dots, K_t$ where $n\le t\le m$ and the latter is true for $K_{t+1}, \dots, K_m$. Of course, if $t=m$ the latter set is empty. Note that $\Phi({\bf g})>K_i$ for all $i>t$. Hence by induction, $\ell({\bf g})\ge c_{m-t}-1$. Also, $2p_1+1, 2p_2+1, \dots, 2p_t+1$ give at least $t-1$ nontrivial, proper factors of $2P+1$ because at most one of them might be 1. Hence $2P+1\ge c_{t-1}$. We now have  
\begin{align*}
\ell({\bf a})&\ge \ell({\bf a}')\\
&\ge \ell({\bf g}^{2P+1}_{r,s})\\
&\ge (2P+1)\ell({\bf g})+P(\ell({\bf r})+\ell({\bf s}))\\
&\ge (2P+1)\ell({\bf g})+2P\\
&\ge (2P+1)(\ell({\bf g})+1)-1\\
&\ge c_{t-1}c_{m-t}-1\\
&\ge c_m-1.
\end{align*} 
\end{proof}

\section{Additional Values of $\text{\bf EK}\boldsymbol{( n )}$}\label{seams}
We begin by determining $\text{\rm EK}(n)$ for $n<45$.  One way to proceed would be to   examine every 2-bridge knot with a given crossing number (by means of computer) to determine the maximum number of strictly smaller nontrivial 2-bridge knots. Presumably this is what Suzuki did to produce the values in (\ref{Suzuki data}). We did this for $n\le 29$ and obtained the same values. Unfortunately, for $n>29$, the time required to examine every 2-bridge knot with crossing number $n$ makes this approach impractical.  

However, Theorem~\ref{main theorem} implies that $\text{\rm EK}(n)<3$ for $n<45$. Thus for each $n<45$, if we simply find one 2-bridge knot whose  crossing number is $n$ and which is strictly greater than two other 2-bridge knots, we will have shown that $\text{\rm EK}(n)=2$. This approach allows us to establish the following theorem, which extends the values of $\text{\rm EK}(n)$ given in (\ref{Suzuki data}).

\begin{theorem}
If $26<n<45$, then $\text{\rm EK}(n)=2$.
\end{theorem}
\begin{proof} In Table~\ref{27 to 44 examples} we list one or more 2-bridge knots for each crossing number $n$ from 27 to 44. It is easy to check that each of these knots is strictly greater than two nontrivial knots by first finding the expanded even sequence and then checking that it parses two different ways. Because $\text{\rm EK}(n)$ cannot be 3 in this range, it must therefore be equal to 2.
\end{proof}

\begin{table}[htp]
\begin{center}
\begin{tabular}{llclc||llcl}
$n$&$p/q$&{\hskip 10pt}&$p/q$&{\hskip 10pt}&$n$&$p/q$&{\hskip 10pt}&$p/q$\\
\hline\hline
& & & & & & \\
 \cline{2-2}\cline{7-7}
27&\multicolumn{1}{|l|}{1/27}&&&&35&\multicolumn{1}{|l|}{1/35}\\
28&\multicolumn{1}{|l|}{17/315}&&&&36&\multicolumn{1}{|l|}{29/595}\\\cline{4-4}\cline{9-9}
29&\multicolumn{1}{|l|}{35/621}&&\multicolumn{1}{|l|}{19/351}&&37&\multicolumn{1}{|l|}{349/5075}&&\multicolumn{1}{|l|}{91/1647}\\\cline{7-7}
30&\multicolumn{1}{|l|}{577/5499}&&\multicolumn{1}{|l|}{35/639}&&38&&&\multicolumn{1}{|l|}{107/1935}\\
31&\multicolumn{1}{|l|}{1189/10395}&&\multicolumn{1}{|l|}{53/945}&&39&&&\multicolumn{1}{|l|}{125/2241}\\\cline{2-2}
32&&&\multicolumn{1}{|l|}{883/8415}&&40&&&\multicolumn{1}{|l|}{2107/20079}\\\cline{2-2}\cline{7-7}
33&\multicolumn{1}{|l|}{1/33}&&\multicolumn{1}{|l|}{1801/15903}&&41&\multicolumn{1}{|l|}{127/2295}&&\multicolumn{1}{|l|}{4249/37935}\\\cline{4-4}\cline{9-9}
34&\multicolumn{1}{|l|}{23/495}&&&&42&\multicolumn{1}{|l|}{143/2583}\\
35&\multicolumn{1}{|l|}{461/5313}&&&&43&\multicolumn{1}{|l|}{161/2889}\\\cline{2-2}
&&&&&44&\multicolumn{1}{|l|}{2719/25911}\\\cline{7-7}
\end{tabular}
\end{center}
\caption{Examples showing $\text{\rm EK}(n)=2$ for $26<n<45$.}

\label{27 to 44 examples}
\end{table}%

The knots given in Table~\ref{27 to 44 examples} appear in six sets, with each set surrounded by a box. In each set, any one of the entries can be used to produce the other entries in the set by means of a construction we call ``negating between seams,''  which we  describe in the next paragraph. Three of the six sets were found by considering the  $(2, 27), (2, 33)$ and $(2, 35)$-torus knots. The other three sets were found by searching 2-bridge knots of a given crossing number until one was found whose expanded even sequence parsed in two ways. That knot was then used to generate the other knots in that box.

To describe this construction, suppose that ${\bf a}, {\bf b}$, and $\bf c$ are all in ${\cal S}_\text{\rm even}$ and that $\bf c$ parses with respect to both $\bf a$ and $\bf b$. A {\it seam} of $\bf c$ is a place to cut $\bf c$ into two pieces so that with respect to each parsing, each piece is composed of a whole number of tiles and connectors. We illustrate the situation using $T(2,27)=K_{1/27}$. The expanded even sequence for this knot, {\bf c}, parses with respect to both {\bf a} and {\bf b} as shown below.  

$${\bf c}=\underbrace{\overbrace{2,-2}^{\bf a}, 2, \overbrace{-2,2}^{{\bf a}^{-1}}, -2,\overbrace{2,-2}^{\bf a}}_{\bf b}, 2,
\underbrace{\overbrace{-2,2}^{{\bf a}^{-1}}, -2, \overbrace{2,-2}^{\bf a}, 2,\overbrace{-2,2}^{{\bf a}^{-1}}}_{{\bf b}^{-1}}, -2,
\underbrace{\overbrace{2,-2}^{\bf a}, 2, \overbrace{-2,2}^{{\bf a}^{-1}}, -2,\overbrace{2,-2}^{\bf a}}_{\bf b} 
$$
There are four seams, located at positions 8, 9, 17, and 18, which cut $\bf c$ into five pieces.
Between any pair of seams, each parsing consists of a whole number of tiles and connectors. Thus, if we negate the portion of $\bf c$ that lies between any two seams (or before the first seam or after the last seam) to obtain a new vector ${\bf d}$, then ${\bf d}$ will still parse with respect to both $\bf a$ and $\bf b$. When going from $\bf c$ to $\bf d$, we will not change the sum of the absolute values of the entries of the vector, but the number of sign changes may change. Thus, by Theorem~\ref{suzuki theorem},  the crossing number of $\Phi(\hat{{\bf d}})$ will differ from $\Phi(\hat{\bf c})$ by the change in the number of sign changes. In this example, the sum of the absolute values of the components of $\bf c$ is 52 and the number of sign changes is 25, which is the most possible. Hence by Theorem~\ref{suzuki theorem}, the crossing number of this 2-bridge knot is 27. Suppose $\bf d$ is obtained from $\bf c$ by negating everything after the last seam. This will give the knot $K_{17/315}$ with crossing number $52-25=28$. Similarly, negating between the third and fourth seams gives $K_{35/621}$, between the second and third and after the fourth gives $K_{577/5499}$, and lastly, between the first and second and between the third and fourth gives $K_{1189/10395}$.

Finally, we can comment on a few values of  $\text{\rm EK}(n)$ for $45\le n\le 105$. Corollary~\ref{EK inverse} implies that $\text{\rm EK}(45)=4$ and that $\text{\rm EK}(105)=6$. The $(2,63)$-torus knot is strictly greater than four torus knots and hence $\text{\rm EK}(63)=4$. Using the $(2,45)$ and $(2,63)$-torus knots, negating between seams give examples that show that $\text{\rm EK}(n)=3 \text{ or } 4$ for $n=46, 47, 48, 49, 64, 65, 66, 67$.

We close with a proof of Theorem~\ref{lower bound theorem}.
\setcounter{theorem}{2}
\begin{theorem} For all $N\ge 3n\ge 9$, we have $\text{\rm EK}(N)\ge \text{\rm EK}(n)$. 
\end{theorem}
\begin{proof} Let $n$ be any natural number and $\bf c$ a vector that has crossing number $n$ and parses $\text{\rm EK}(n)$ ways. We will use $\bf c$ to build a vector $\bf d$ that has any crossing number $N\ge 3n$ and which also parses in as many ways as {\bf c}. This will give that $\text{\rm EK}(N)\ge \text{\rm EK}(n)$.

 If $N-3n$ is even, choose $m$ so that $|m|=N-3n$ and, if not zero, $m$ has the same sign as the last entry of ${\bf c}$.   Let ${\bf d}=({\bf c}, {\bf m}, {\bf c}^{-1}, 0, {\bf c})$. Using Theorem~\ref{suzuki theorem}, we find that the crossing number of $\bf d$ is $3n+|m|=N$.  If $N-3n$ is odd, then let $m=N-3n+1$ and  ${\bf d}=({\bf c}, {\bf m}, -{\bf c}^{-1}, 0, -{\bf c})$. Theorem~\ref{suzuki theorem} now implies the crossing number of $\bf d$ is $3n+m-1=N$. In either case, if $\bf c$ parses with respect to $\bf a$, then so does $\bf d$.  
\end{proof}
 

\end{document}